\documentclass[letterpaper, 10 pt, conference]{ieeeconf}  % Comment this line out if you need a4paper

\usepackage{amssymb,latexsym,amsfonts,amsmath} 
\usepackage{dsfont}

\IEEEoverridecommandlockouts                              % This command is only needed if 
\usepackage{acronym}
\usepackage{pgfgantt}
\usepackage[utf8]{inputenc}
\usepackage{amsmath}
\usepackage{amssymb}
\usepackage[final]{pdfpages}
\usepackage{amsthm}
\usepackage{mathtools}
\usepackage{subcaption}
\usepackage{romannum}
\usepackage{wrapfig}
\usepackage{todonotes}
\usepackage{siunitx}
\usepackage{lineno}
\usepackage{bbm}
\usepackage{tikz}
\usepackage{graphicx}
\usepackage{color}
\usepackage{url}
\usepackage{multirow}
\usepackage{comment}
\usepackage{paralist}
\usepackage{scrextend}
\usepackage{amsfonts}

\newcommand{\ie}{{\it i.e.}}

\newcommand{\etal}{{et~al.}}

\newcommand{\reals}{\mathbb{R}}

\newcommand{\Expect}{\mathop{\bf E{}}}

\newcommand{\norm}[1]{\lVert #1 \rVert}

\newcommand{\nat}{\mathbb{N}}

\newcommand{\Next}{\mathsf{X}}
\newcommand{\Eventually}{\lozenge}
\newcommand{\until}{\mathsf{U}}
\newcommand{\true}{\mathsf{True}}
\newcommand{\supp}{\mathsf{Supp}}

\newcommand{\abs}[1]{\lvert#1\rvert}

\newcommand{\calAP}{\mathcal{AP}}

\newcommand{\calB}{\mathcal{B}}

\newcommand{\calK}{\mathcal{K}}

\newcommand{\indicator}{\mathbf{1}}

\newcommand{\bs}{\mathbf{s}}

\acrodef{rl}[RL]{Reinforcement Learning}
\acrodef{mdp}[MDP]{Markov Decision Process}
\acrodef{ggks}[GGKs]{Geodesic Gaussian Kernels}
\acrodef{ogks}[OGKs]{Ordinary Gaussian Kernels}
\acrodef{sip}[SIP]{Semi-Infinite Programming}

\acrodef{pctl}[PCTL]{Probabilistic Computation Tree Logic}
\acrodef{mc}[MC]{Markov Chain}
\acrodef{lp}[LP]{Linear Programming}
\acrodef{pcl}[PCL]{Path Consistency Learning}
\acrodef{lp}[LP]{Linear Programming}
\acrodef{adp}[ADP]{Approximate Dynamic Programming}
\acrodef{alp}[ALP]{Approximate Linear Programming}
\acrodef{elp}[ELP]{Exact Linear Programming}
\acrodef{ialp}[iALP]{Iterative Approximate Linear Programming}
\acrodef{msbr}[MSBR]{Mean Square Bellman Residual}
\acrodef{dfa}[DFA]{Deterministic Finite-state Automaton}
\acrodef{lqr}[LQR]{Linear Quadratic Regulator}
\acrodef{hjb}[HJB]{Hamilton-Jacobi-Bellman}
\acrodef{kl}[KL]{Kullback-Leibler}
\acrodef{pac}[PAC]{Probably Approximately Correct}
\acrodef{ipi}[iPI]{Iterative Policy Iteration}
\acrodef{cmdps}[CMDPs]{Constrainted Markov Decision Processes}
\acrodef{pomdps}[POMDPs]{Partially Observable Markov Decision Processes}
\acrodef{decmdps}[DECMDPs]{Decentralized Markov Decision Processes}
\acrodef{dp}[DP]{Dynamic Programming}
\acrodef{cdf}[CDF]{Cumulative Distribution Function}
\acrodef{pmf}[PMF]{Probability Mass Function}
\acrodef{decpomdps}[DECPOMDPs]{Decentralized Partially Observable Markov Decision Processes}
\acrodef{ltl}[LTL]{Linear Temproal Logic}

\theoremstyle{definition}
\newtheorem{definition}{Definition}[section]
\newtheorem{theorem}{Theorem}[section]

\newtheorem{lemma}[theorem]{Lemma}

\newtheorem{remark}{Remark}
\newtheorem{problem}{Problem}

\begin{document}

\title{\Large \bf Approximate Dynamic Programming with Probabilistic Temporal Logic Constraints}

\author{Lening~Li and Jie~Fu \thanks{L. Li and J. Fu are with the Robotics Engineering Program, Department of Electrical and Computer Engineering, Worcester Polytechnic Institute, Worcester, MA, 01609, USA, {\tt\small lli4, jfu2@wpi.edu}}}

\maketitle

\begin{abstract}
In this paper, we develop approximate dynamic programming methods for stochastic systems modeled as Markov Decision Processes, given both soft performance criteria and hard constraints in a class of probabilistic temporal logic called Probabilistic Computation Tree Logic (PCTL). Our  approach consists of two steps: First, we show how to transform  a class of PCTL formulas into chance constraints that can be enforced during planning in stochastic systems. Second, by integrating randomized optimization and  entropy-regulated dynamic programming, we devise a novel trajectory sampling-based approximate value iteration method to iteratively solve for an upper bound on the value function while ensuring the constraints that PCTL specifications are satisfied. Particularly, we show that by the on-policy sampling of the trajectories, a tight bound can be achieved between the upper bound given by the approximation and the true value function. The correctness and efficiency of the method are demonstrated using robotic motion planning examples. 
\end{abstract}

\section{Introduction}
For safety-critical systems, one main control objective is to ensure desirable system performance with provable correctness guarantees given high-level system specifications. 
In this work, we study the following problem: Given a stochastic system modeled as a \ac{mdp}, how to \emph{efficiently synthesize} a policy that is optimal with respect to a performance criterion while satisfying safety- and mission-critical constraints in temporal logic?  For stochastic systems with difficult-to-model dynamics, is it possible to \emph{learn} such a policy efficiently from sampled trajectories using a black box physics simulator of the system?

\ac{mdp} planning with temporal logic constraints has been extensively studied. One class of research develops policies that maximize the probability of satisfying given temporal logic specifications \cite{ding2014optimal, Fu-RSS-14,wang2015temporal,fu2015computational,lahijanian2012temporal,lahijanian2011control,sadigh2014learning}. Another class of work considers multiple objectives, including soft constraints---maximizing the total reward---and hard constraints---satisfying safety properties \cite{dimitrova2016robust,junges2016safety}. Among these work,  Wang \etal \cite{wang2015temporal} devised the first \ac{adp} method to solve the problem of maximizing the probability of satisfying temporal logic constraints. The principle of \ac{adp} is to introduce policy function approximation, value function approximation, or both (actor-critic methods) \cite{bertsekas1996dynamic} such that the number of decision variables originally grows linearly in the state space and action space of the \ac{mdp} but becomes independent of the size of the \ac{mdp}. Thus, \ac{adp} methods \cite{bertsekas2008neuro} allow one to address the problem of scalability and the lack of a dynamic model in planning for stochastic systems.

In this work, we develop approximate value iteration---a class of \ac{adp} methods---in \ac{mdp} given both soft performance criteria and hard temporal logic constraints. The hard constraints are given by \ac{pctl} \cite{kwiatkowska2007stochastic}, which is used to reason about properties in stochastic systems, for example, ``the probability is greater than 0.85 that the goal can be reached with a cost less than 100''. Our approach includes two steps: In the first step, we show that a large subclass of \ac{pctl} can be equivalently represented by chance constraints over the path distribution in a stochastic system, with appropriately defined cost functions. We introduce the mixing time for Markov chain to approximately verify properties in time-unbounded \ac{pctl} using trajectories with finite lengths. It is noted that \ac{pctl} involves global properties in a system, which is often hard to enforce using locally optimal policy search methods. In the second step, we develop a chance-constrained approximate value iteration method to solve the planning problem with reward maximization and \ac{pctl} constraint satisfaction. In literature, chance-constrained approximate policy iteration method has been developed \cite{chow2018risk}. We consider approximate value iteration methods to obtain  guarantees for global properties in \ac{pctl}. This is achieved via integrating randomized optimization and approximate linear programming formulation proposed in \cite{de2003linear}. Our sampling-based approximate value iteration method has the following desirable properties: \begin{inparaenum}\item It achieves a tight error bound by weighing the approximation errors over the state space using the state visitation frequencies of an approximately optimal policy; \item It is probabilistic complete, in the sense that it converges to an approximately optimal policy that satisfies the \ac{pctl} constraints with probability one. \end{inparaenum}

The rest of the paper is structured as follows. Section~\ref{sec:preliminaires} provides some preliminaries. Section~\ref{sec:main_result} contains the main results of the paper, including the translation from a subclass of \ac{pctl} to chance constraints in \ac{mdp} and the \ac{adp} algorithm. Presented in section~\ref{sec:exp} are experimental studies with robot motion planning to validate the optimality and correctness of the proposed method. Section~\ref{sec:con} draws a conclusion.

\section{Preliminaries}\label{sec:preliminaires}
Notation: A finite set $X$, $\abs{X}$ is the size of $X$, and $\Delta(X)$ denotes the probability simplex in $\reals^{\abs{X}}$. Given a distribution $\mu \in \Delta(X)$, $\supp(\mu) =\{x\in X\mid \mu(x)\ne 0\}$ is the \emph{support} of $\mu$. $\reals$, $\reals_+$, $\reals_{\ge 0}$, and $\nat$ are the set of reals, positive reals, nonnegative reals, and natural numbers, respectively. $\calAP$ is the set of the atomic propositions.
\subsection{Markov decision processes and the linear programming formulation of \ac{adp}}
We consider stochastic systems modeled by Markov Decision Processes. An \ac{mdp} is a tuple $ M = \langle S, A, P, r, \mu,\gamma \rangle$, where $S$ is a finite set of  states; $A$ is a finite set of  actions;  $P(\cdot \mid s,a) \in \Delta(S)$ is the probability distribution of the next state by taking action $a$ at state $s$;   $r: S \times A \rightarrow \reals$ is the reward function: $r(s,a)$ represents the immediate reward gained at state $s$ taking action $a$; $\mu \in \Delta(S) $ is the initial state distribution;	$\gamma \in (0,1]$ is a discount factor. For $s\in S$, we denote $A(s) =\{a\in A\mid \exists s'\in S, P(s' \mid s,a)>0\}$ the set of \emph{admissible} actions at state $s$.

Given a time step $t \in \nat$, a \emph{history} $h_t = s_0, a_0, s_1,\allowbreak a_1,\ldots s_t$ is a finite sequence of state-action pairs prior to time $t$. Let $H_t$ be the set of all possible histories prior to time $t$. A decision rule $d_t: H_t \rightarrow \Delta(A)$ maps a history $h \in H_t$ to a distribution $d_t(h)$ of admissible actions. Let $\pi = (d_1,d_2,\ldots, d_N)$ denote a randomized history-dependent policy. If $d_i=d_j=d$ for all $i\ne j$ in a policy, then $\pi = (d,d,\ldots, d)$ is a Markovian, randomized policy. We denote $\Pi$ the set of Markov and randomized policies. A policy $\pi $ induces, from an \ac{mdp} $M$, a Markov chain $M^\pi = X_0, A_0, X_1, A_1, X_2, \ldots$, where $X_i,A_i$ are the random variables describing the $i$-th state and action in the chain. We also omit the actions and refer to the $\pi$-induced Markov chain by $M^\pi = \{X_n, n\ge 0\}$. It holds that $X_0 \sim \mu$, $P(A_t=a_t|X_t=s_t) = \pi(a_t|s_0\ldots s_t)$, and $P(X_{t+1} = s_{t+1} \mid X_t = s_t) = \sum_{a \in A} P(s_{t+1} \mid s_t, a_t) \pi(a_t \mid s_0\ldots  s_t)$. We say a state $s\in S$ is a \emph{sink} state if $P(s \mid s,a)=1$ for all $a\in A(s)$.

A path $\rho=s_0s_1s_2\ldots$ is a finite/infinite sequence of
states. Given a policy $\pi$ and path $\rho$, we denote $P^\pi (\rho)$
the probability of the path $\rho$ in the $\pi$-induced Markov
chain. We denote $\mbox{Path}^\pi(s)$ the set of possible paths
starting from state $s$ and following policy $\pi$.

Next, we present the \ac{lp} formulation for optimal planning in \ac{mdp}s \cite{de2003linear}. In planning to maximize an infinite-horizon
discounted reward, we have a value function of a policy defined by
\begin{equation*}
V^{\pi}(\mu) = \Expect  \left[ \sum_{t=0}^{\infty}\gamma^{t}r(X_t, A_t) \mid X_{0}  \sim \mu \right].
\label{eq:criteria}
\end{equation*}
The optimal policy $\pi^\ast$ achieves $\pi^\ast =\arg\max_{\pi \in \Pi} V^\pi(\mu)$. For maximizing an infinite-horizon discounted reward, there exists an optimal, memoryless policy \cite{geibel2006reinforcement}. The \ac{lp} \cite{puterman2014markov} solving $\pi^\ast$ is described as follows: Let $V = [V(s)]_{s\in S}$ be a vector of variables, one for each state.
\begin{equation}
\label{eq:lp}
\begin{split}
\min_{V} \; & c^{\intercal}V,  \\
\mbox{subject to: } & \left[r(s, a) + \gamma \sum_{s'}P(s' \mid s, a)V(s')\right] \leq V(s),  \\
& \forall s \in S, a \in A,
\end{split}
\end{equation}
where $c= [c_1,\ldots, c_{\abs{S}}]^\intercal$ is a vector of nonnegative state-relevance weights, \ie, $c_i\ge 0$, for all $i=1,\ldots, \abs{S}$. Once $V$ is obtained, the optimal policy can be generated using the Bellman equation.

For large-scale \ac{mdp}s, function approximation of $V$ is introduced to find approximate-optimal policy  \cite{de2003linear}. Consider a linear function approximator
\begin{equation*}
V(s;\theta) \approx \sum_{k=1}^{\calK}\phi_{k}(s) \theta_{k} = \Phi\theta,
\label{eq:state_value}
\end{equation*}
where $\phi_{k}(s): S \rightarrow \reals, k = 1, \dots, \calK$ are preselected basis functions and $\theta = [\theta_1, \ldots, \theta_{\calK}]^\intercal \in \reals^{\calK}$ is a weight vector. The approximate \ac{lp} is to substitute $V$ with its function approximator $\Phi\theta$ in \eqref{eq:lp}.
Let $\theta^\ast$ be the solution. An approximate-optimal policy is computed by $\pi(s; \theta^\ast) \coloneqq \arg\max_a \left[r(s, a) + \gamma \sum_{s'}P(s' \mid s, a)V(s';\theta^\ast)\right]$.

\subsection{Specification: Probabilistic Computation Tree Logic (PCTL) with the reachability reward/cost properties} \ac{pctl} provides syntax and semantics to quantify probabilistic properties~\cite{fagin1990logic}.

The syntax of \ac{pctl} with reachability reward/cost properties \cite{kwiatkowska2007stochastic} is defined as follows:
\begin{equation*}
\resizebox{0.5\textwidth}{!}{$
\begin{array}{cccccccccc}
\phi                            & \coloneqq                                     & \multicolumn{1}{c|}{\true} & \multicolumn{1}{c|}{\alpha} & \multicolumn{1}{c|}{\phi \land \phi} &     
\multicolumn{1}{c|}{\neg \phi} & \multicolumn{1}{c|}{P_{\bowtie p}[\psi]} & {C_{\bowtie m}(\Eventually^{\le k } \phi)} ; \\
\\
\psi & \multicolumn{1}{c}{\coloneqq} & 
\multicolumn{1}{c|}{\Next \phi} & \multicolumn{1}{c|}{\phi \until^{\le k} \phi} & {\phi \until \phi} ,        &                             &
\end{array}	$
}
\end{equation*}
where $\phi$ is a state formula and $\psi $ is a path formula,
$\alpha \in \calAP$ is an atomic proposition, $k \in \nat$ is a nonnegative integer, $\bowtie \in \{\le, \ge, >, <\}$, $p\in [0,1]$ is a
probability, and $m \in \reals$. A path formula $\psi$ is interpreted on paths.  $\Next$ is the \emph{next} operator, $\until$ is the \emph{until operator}, $\until^{\le k}$ is the \emph{bounded until} operator, and
$\Eventually \phi \equiv \true \until \phi$ is ``eventually''. $\Next \phi $ asserts that the next state satisfies a state formula $\phi$. $\phi_1 \until^{\le k} \phi _2$ asserts that $\phi_2$ is satisfied within $k$ steps and all preceding states satisfy $\phi_1$.  $\phi_1 \until \phi _2$ asserts that $\phi_2$ is satisfied some time in the future and all preceding states satisfy $\phi_1$.  $\Eventually^{\le k} \phi$ means $\phi$ becomes true in no more than $k$ steps. $P_{\bowtie p}[\psi]$ means that the probability of generating a trajectory that satisfies formula $\psi$ is $\bowtie p $. The reachability reward formula
\cite{kwiatkowska2007stochastic} $C_{\bowtie m}[\Eventually^{\le k} \phi]$ means that the total accumulated rewards/costs along a path of length no greater than $k$ that reaches a state that satisfies $\phi$ are $\bowtie m$, for a predefined reward/cost function. Note that we only
consider time-bounded total cost.

Formally, given a $\pi$-induced Markov chain $M^\pi $ and a \ac{pctl} formula $\phi$, the policy satisfying the constraint $\phi$ given the initial state $s$ is denoted $s \models_\pi \phi$.

For ease of understanding, we give several examples of \ac{pctl} formulas
\begin{itemize}
\item $ P_{\ge 0.95}(\true \until^{\le 10} \mbox{reach goal})$: From any state, the goal can be reached in less than $10$ steps with probability at least $0.95$.
\item $P_{> 0.9}(\Next C_{\le 100} [\Eventually^{\le 5} \alpha])$: ``The probability is larger than 0.9, that from the next step onward, the system reaches a state satisfying $\alpha$ with a total accumulated cost no greater than $100$ in no more than 5 steps.
\end{itemize}

Our goal is to develop an approximate dynamic programming algorithm that solves the following planning problem:
\begin{problem}
\label{problem1}
Given an \ac{mdp} $M=\langle S, A, P, r, d, \mu, \gamma \rangle$ and a \ac{pctl} formula $\phi$, the goal is to find a policy $\pi$ that solves
\[\max_{\pi} V^\pi( \mu), \text{ subject to }
M^\pi \models \phi.
\]
\end{problem}

\section{Main result}
\label{sec:main_result}
Our approach to tackling Problem \ref{problem1} includes two steps: First, we show that a class of \ac{pctl} constraints can be translated into chance constraints. Second, we introduce an approximate value iteration algorithm that solves Problem~\ref{problem1}.

\subsection{Translating \ac{pctl} formulas into chance constraints}
\label{sec:translate}
Chance constraints are introduced to capture \emph{risk-sensitive} optimization criteria in \ac{mdp}s. Let $d:S\times A\rightarrow \reals$ be a cost function. For a policy $\pi$, we define the cost of a state $s$ (resp. state-action pair $(s, a)$) as the sum of finite-horizon (discounted) total costs encountered by the decision-maker when it starts at state $s$ (resp. state-action pair $(s, a)$) and then follows policy $\pi$, as follows:
\[
D(s,a,T;\pi) = \sum_{t=0}^{T-1} \gamma^t d(X_{t}, A_{t})
\mid X_{0} = s, A_{0} = a ; 
\]
\[
D(s,T;\pi) = \sum_{a\in A(s)}D(s,a,T;\pi)\pi(a \mid s),
\]
where $\gamma$ is a discount factor and $T \in \nat \cup\{\infty\}$ is a  stopping time.

Chance-constrained planning in \ac{mdp} aims to  ensure that for a given initial state $X_0$, a confidence level $\beta \in (0, 1)$, and cost tolerance $\alpha$, the policy $\pi$ that satisfies
\begin{equation*}
\Pr(D(X_0,T;\pi) \ge \alpha)\le \beta.
\end{equation*}

Let $\rho=s_0s_1\ldots \in \mbox{Path}^\pi(s)$, we define $
D(\rho;\pi) = \sum_{t=0}^{\abs{\rho}-1} \gamma^t  d(s_{t}, a_{t}).$
Then the chance constraint is equivalently expressed as 
\[\Pr(D(\rho,T;\pi) \ge \alpha)\le \beta, \rho \sim \{X_n, n\ge 0\},
\]
where $\{X_n, n\ge 0\}$ is the  $\pi$-induced Markov chain $M^\pi$. 

We introduce the notion of mixing time and use it later to determine a stopping time $T$ for approximately satisfying  time-unbounded specifications in \ac{pctl}.

\begin{definition}[Mixing time]
Let $M$ be an \ac{mdp} and let $\pi$ be an ergodic policy in $M$. Given $\epsilon \in \reals_+$, the $\epsilon$-return mixing time of $\pi$ is the smallest $T$ such that for all $T'\ge T$ , $\abs{D(s, T'; \pi)  - D(s,T;\pi)} \le \epsilon$.
\end{definition}
Given a policy $\pi$, we can find an upper bound of this mixing time in terms of the second eigenvalue of the state transition $P^\pi(s' \mid s)$ matrix of the Markov chain $M^\pi$ using methods in \cite{sinclair2012algorithms}.

Next, we show how to translate \ac{pctl} formulas into chance constraints. We select the discount factor
$\gamma=1$
unless otherwise specified. The reason is that \ac{pctl}
considers the probability of satisfying path formulas and total cost without discounting. We distinguish three classes of \ac{pctl}.

\subsubsection{Probabilistic formula $P_{\bowtie p}(\psi)$}
We consider a class of probabilistic formula $P_{\bowtie
p}(\psi)$ where $\psi$ is a path formula of the form
$\neg \phi_1 \until \phi_2$, $\true \until \phi$,
or $\Next \phi$, where $\phi,\phi_1, \text{ and } \phi_2$ are propositional
logic formulas. Formally,
\[
s\models_{\pi} P_{\bowtie p}(\psi) \mbox{ iff } P(\{\rho \in \mbox{Path}^\pi(s)) \mid \rho \models \psi\}) \bowtie p.
\]
We show that a formula in this class can be represented by a chance constraint with a properly defined cost function.

\begin{lemma}
\label{lm:next}
Given a formula $P_{\bowtie p}(\Next \phi)$, let's define a
cost function $d:S\times A\rightarrow \reals$ as
$d(s,a) =\Expect_{s' \sim P(\cdot \mid s,a)}
\indicator(s'\models \phi)$.  Given a policy $\pi$, the $\pi$-induced Markov
chain satisfies the formula, denoted
$M^\pi \models P_{\bowtie p}(\Next \phi)$, if
and only if $D(X_0,T;\pi) \bowtie p$ and the stopping time
$T=2$.
\end{lemma}

\begin{lemma}
\label{lm:trueuntil}
Given a formula $P_{\bowtie p}(\true \until \phi)$, let's
define a cost function $d:S\times A\rightarrow \reals$ as
$d(s,a) =\Expect_{s' \sim P(\cdot \mid s,a)}\indicator(s'\models
\phi)$ only if $s\not \models
\phi$.  Let all states in the set $\{s \mid s\models\phi\}$ be
sink states. Given a policy $\pi$, the $\pi$-induced Markov chain
satisfies the formula, denoted
$M^\pi \models P_{\bowtie p}(\true \until \phi)$, if it is one
of the following cases:
\begin{itemize}
\item $\bowtie \in \{\ge, >\}$: $ D(X_0,T_\epsilon;\pi) \bowtie p  $.
\item $\bowtie \in \{\le, < \}$: $ D(X_0,T_\epsilon;\pi) \bowtie p -  \epsilon $.
\end{itemize}
where $\epsilon \in (0,1)$ is a small constant and $T_{\epsilon}$ is an upper bound of the $\epsilon$-mixing time of $\pi$.
\end{lemma}
\begin{proof}
By definition of cost function $d$, $D(X_0; \pi) $ is the probability of eventually reaching a state satisfying $\phi$; $D(X_0, T; \pi) $ is the probability of reaching a state satisfying $\phi$ within $T$ steps. Given $T_\epsilon$ is the mixing time, if a path of length $T_\epsilon$ has not yet visited a state that satisfies $\phi$, then the probability of satisfying the path formula as we continue along this path is less than $\epsilon$. Since the cost is nonnegative, we have $D(X_0 ;\pi) - D(X_0, T_\epsilon; \pi) \le \epsilon$ for a predefined positive real number $\epsilon$. Next, we consider two cases: Case I: $\bowtie \in \{\ge, >\}$, $P_{\ge p}(\true \until \phi)$ is equivalent to $D(X_0 ;\pi) \ge p$. Given $D(X_0 ;\pi) \ge D(X_0, T_\epsilon; \pi)$, a sufficient condition for $D(X_0 ;\pi) \ge p$ is that $D(X_0, T_\epsilon;\pi) \ge p$. The same argument applies for strictly greater than, \ie, $>$. Case II: $\bowtie \in \{\le, <\}$,  $P_{\le p}(\true \until \phi)$ is equivalent to $D(X_0 ;\pi) \le p$. Given $D(X_0 ;\pi)  - D(X_0, T_\epsilon; \pi) \le \epsilon   $, we have $D(X_0 ;\pi) \le D(X_0, T_\epsilon; \pi)  +\epsilon$. A sufficient condition for $D(X_0 ;\pi) \le p$ is that  $D(X_0, T_\epsilon; \pi)  +\epsilon \le p$, which is equivalent to $D(X_0, T_\epsilon; \pi)   \le p-\epsilon$.  The same argument applies for strictly less than, \ie, $<$.
\end{proof}

\begin{lemma}
\label{lm:notuntil}
Given a formula $P_{\bowtie p}(\neg \phi_1 \until \phi_2)$,
let's define a cost function $d:S\times A\rightarrow \reals$
as
$d(s,a) =\Expect_{s' \sim P(\cdot \mid s,a)}
\indicator(s'\models \phi_2)$ only if
$s\not \models \phi_2$.  Let all states in the set
$\{s \mid s\models\phi_1\lor \phi_2\}$ be sink states. Given a
policy $\pi$, the $\pi$-induced Markov chain satisfies the
formula, denoted
$M^\pi \models P_{\bowtie p}(\neg \phi_1 \until \phi_2)$, if
it is one of the following cases:
\begin{itemize}
\item $\bowtie \in \{\ge, >\}$: $ D(X_0,T_\epsilon;\pi) \bowtie p  $.
\item $\bowtie \in \{\le, < \}$: $ D(X_0,T_\epsilon;\pi) \bowtie p -  \epsilon $.
\end{itemize}
%	where the stopping time is defined by
%	\[
%	T = \min (\min_t(\indicator(X_t\models \phi_2),  \indicator(X_t\models \phi_1)), T_{\epsilon})
%	\]
where $T_{\epsilon}$ is an upper bound of the $\epsilon$-mixing time  of $\pi$.
\end{lemma}
\begin{proof}
The proof is similar to that of Lemma~\ref{lm:trueuntil}. It is noted that when a system reaches a state satisfying $\phi_1$ prior to reaching a state satisfying $\phi_2$, it receives a cost of $0$ even if the path continues to evolve until the time bound $T_\epsilon$.
\end{proof}
\subsubsection{Risk neutral cost constraint $C_{\bowtie m} \Eventually^{\le k} \phi$}
$C_{\bowtie m}\Eventually^{\le k} \phi$ is the expected cost
of paths satisfying $\Eventually^{\le k} \phi$.  Based on the
definition in \cite{kwiatkowska2007stochastic}, a policy $\pi$
satisfying the constraint given an initial state $s$, denoted
$ s\models_\pi C_{\bowtie m} \Eventually ^{\le k}\phi $, if
\[
\Expect\limits_{\pi}(s, X_{\Eventually ^{\le k}\phi} ) \bowtie m,
\]
where $\Expect\limits_{\pi}(s, X_{\Eventually ^{\le k}\phi} )
$ denotes the expectation of the random variable
$ X_{\Eventually ^{\le k}\phi} : \mbox{Path}^\pi(s)
\rightarrow \reals_{\ge 0}$ with respect to the Markov chain
$M^\pi$.
Let $d:S\times A\rightarrow \reals$ be defined by
$d(s,a) =\Expect_{s' \sim P(\cdot \mid s,a)} \indicator(s'\models
\phi)$.  
For any state-action sequences 
$\rho=s_0a_0 s_1a_1\ldots \in  \mbox{Path}^\pi(s) $,
$ X _{\Eventually ^{\le k}}(\rho) = D(\rho; \pi) =
\left(\sum_{t=0}^{T} d(s_{t}, a_{t})\right) + \bar D
\indicator(s_T\not \models \phi), $ where
$T=\min(k, \min\{j |s_j\models \phi\})$ and $\bar D \gg 0$ is
a penalty term, which is added when the path does not satisfy
the specification $\Eventually ^{\le k}\phi$. Note that
instead of assigning a cost of $\infty$ to a path that fails
to satisfy $C_{\bowtie m}\Eventually^{\le k} \phi$ in
\cite{kwiatkowska2007stochastic}, we use a large penalty to
ensure the planning problem is well-defined.

\subsubsection{Risk sensitive \ac{pctl}} We consider a class of \ac{pctl} of the following from: $\phi_1 \implies  \Pr_{\bowtie_1 p}( \Next C_{\bowtie_2 m} \Eventually^{\le k}\phi_2)$ where $\phi_1$ and $\phi_2$ are propositional logic formulas, \ie, state formulas that use only conjunction and negation with atomic propositions in $\calAP$. The \emph{risk-sensitive} semantics of the formula means that if a Markov chain satisfies the formula, then starting from  a state that satisfies $\phi_1$, in the next step the probability of a sampled path that satisfies $\Eventually^{\le k} \phi_2 $ with a cost $\bowtie_2 m$ is $\bowtie_1 p$. Formally, it is defined by
\begin{multline*}
s\models_{\pi} \left( \phi_1 \implies  \Pr_{\bowtie_1 p}( \Next C_{\bowtie_2 m} \Eventually^{\le k}\phi_2)\right)
\mbox{ iff } s\not \models \phi_1\\
\text{ or } s \models \phi_1 \text{ and } \Pr (\Next  C_{\bowtie_2 m}\Eventually^{\le {k}} \phi_2) \bowtie_1 p.
\end{multline*}

\begin{lemma} Given $\phi \coloneqq \phi_1 \implies  \Pr_{\bowtie_1 p}( \Next C_{\bowtie_2 m} \Eventually^{\le k}\phi_2)$, let's define a cost function $d:S\times A\rightarrow \reals$ as $d(s,a) =\Expect_{s' \sim P(\cdot \mid s,a)} \indicator(s'\models \phi_2)$. 
Given a policy $\pi$, the $\pi$-induced Markov chain satisfies the formula, denoted $M^\pi \models  \phi$, if  
\[
P(D(s, k+1;\pi) \bowtie_2 m ) \bowtie_1 p, \quad \forall s \models \phi_1,
\]
where $D(s, k+1;\pi) 
=\left(\sum_{t=1}^{T} d(s_{t}, a_{t})\right)  + \bar D \indicator(s_T\not \models \phi_2)$, $T=\min(k+1, \min\{j |s_j\models \phi_2\})$, and $\bar D \gg 0$ is a penalty.
\end{lemma}
The proof is by construction and omitted.
Further, let $Y=\{s\mid s\models \phi_1\}$, the constraint is 
\[
P(D(s,k+1;\pi) \bowtie_2 m)\bowtie_1 p, \mbox{ for all } s\in Y.
\]

% Without loss of generality, in the following, we presented the
% \ac{adp} for chance constrained planning using value function
% approximation. Ref. \cite{chow2018risk} presented a chance
% constrained \ac{adp} for policy iteration. Instead of policy
% iteration, we propose to use value iteration to obtain a
% global solution that allows us to reason about chance
% constraints that are only triggered in a fraction of the state
% space, different from the initial distribution.

% 	\begin{lemma}
% 		Given a formula $\Pr_{\bowtie_1 p}(\Next C_{\bowtie_2 m} \true \until^{\le k} \phi)$, a cost function $d:S\times A\rightarrow \reals$ defined by $d(s,a) =\Expect_{s' \sim P(\cdot|s,a)} \indicator(s'\models \phi)$, and a policy $\pi$,  the $\pi$-induced \ac{mc} $M^\pi \models  \Pr_{\bowtie_1 p}(\Next C_{\bowtie_2 m} \Eventually \phi)$, if and only if $ \Pr(D(X_0,T;\pi) \bowtie_2 m) \bowtie_1 p $
% 		given that  $D(X_0,A_0, T;\pi) = \sum_{t=\mathbf{1}}^{T-1} \gamma^t  d(X_{t}, A_{t}) \mid X_{0} = s, A_{0} = a$ and $D(X_0,T;\pi) =  \sum_{a\in A(s)}D(s,a,T;\pi)\pi(a|s)$, and $T  = \min\{\{t\mid X_t\models \phi\}, k\}$.
% 	\end{lemma}
% 	\begin{proof}
% 	\end{proof}

So far, we have shown that any \ac{pctl} formula in the three subclasses  can be represented as a chance constraint with an appropriate definition
of the cost function $d:S\times A\rightarrow \reals$ of the following form:	for some $\alpha \in \reals$ and $\beta \in [0,1]$, $k$
\begin{equation*}
\Pr(D(s ,k; \pi) \ge \alpha )\le \beta\;,  \quad \forall s\in Y,
\end{equation*}
where $k$ is a positive integer and $Y\subseteq S$. 

\begin{remark} It is noted  $D(s,k;\pi) < \alpha$ is a special case of chance constraint and can be expressed by $\Pr(D(s ,k; \pi) \ge \alpha )\le 0$. 
The conjunction of multiple \ac{pctl}s can be translated into multiple chance constraints. 
\end{remark}

\subsection{Formulating stochastic programming for \ac{pctl} constrained optimal planning}
\label{section:adp} 
To develop  an \ac{adp} method, we first replace the hardmax operator with the softmax operator~\cite{sutton1998reinforcement}, defined by
\begin{multline}
\calB V(s) =     \\                                                                                        
\tau \log \sum_{a} \exp\left\{\left(r(s, a) + \gamma \sum_{s'}P(s' \mid s, a)V(s')\right) / \tau\right\},	\label{eq:softmax_backup}                                                                                			\end{multline} 
where $\tau > 0$ is a predefined temperature parameter. With the $\tau$ approaches $0$, Eq.~\eqref{eq:softmax_backup} recovers the hardmax Bellman operator. The softmax Bellman operator is contracting \cite{sutton1998reinforcement}. Given $V(\cdot; \theta)$ the value function parameterized by $\theta$, one can obtain the corresponding state action function $Q(\cdot; \theta)$ and policy $\pi(\cdot; \theta)$ as:
\begin{subequations}	\label{eq:Q-policy}
\begin{align}
& Q(s,a; \theta) = r(s, a) + \gamma \sum_{s'}P(s' \mid s, a) V(s'; \theta) \label{eq:Q}, \\
& \pi(a \mid s; \theta) = \exp((Q(s,a; \theta)- V(s; \theta))/\tau). \label{eq:policy}
\end{align}
\end{subequations}
We denote $V^\ast$ the fixed point of the softmax operator, \ie, $\calB V^\ast =V^\ast$.

Due to the monotonic contraction property of the softmax Bellman operator $\calB$, it can be shown that for any value function $V$ that satisfies $\calB V \le V$, $V$ is an upper bound of the value function $V^\ast$. 	Given a  value function approximation $V(s;\theta)=  \Phi(s)\theta$, the goal is to search for a function parameter $\theta \in \reals^\calK $ that solves the following optimization problem:
\begin{subequations}
\label{eq:adp}
\begin{align}
\min_\theta         & \sum_{s \in S} c(s)V(s;  \theta), \nonumber                                                 \\
\mbox{subject to: } & \calB V(s;\theta) - V(s; \theta) \leq 0 , \quad \forall s\in S,  \label{eq:value}           \\
& \Pr(D(s, k; \theta) \geq \alpha ) - \beta \leq 0 ,  \quad \forall s\in Y. \label{eq:chance} 
\end{align}
\end{subequations}

Next, we devise a trajectory-based value iteration algorithm for solving chance-constrained \ac{mdp}s. Our approach is based on randomized optimization \cite{tempo2012randomized} that iteratively searches for, in the value function parameter space, an optimal parameter that minimizes a weighted distance between the upper bound given by the value function approximation and the true softmax value function.

% By integrating statistical model checking, we will later show the proposed method is a model-free reinforcement learning algorithm given \ac{pctl} constraints.

We first introduce a continuous function $B: \reals \rightarrow \reals_{+}$  with support equal to 
$(0, \infty)$, in the sense that
\begin{align*}
& B(x) = 0 \text{ for all } x \in (-\infty, 0]  \text{, and }                                \\
& B(x) >0  \text{ for all } x \in (0, \infty).
\end{align*}
One such function is $B(x) = \max\{x,0\}$. 
Let 
\[ g(s; \theta) =\calB V(s; \theta)- V(s; \theta), \forall s\in S.\] 
Then, the constraints in \eqref{eq:value} become $g(s; \theta) \le 0 $ for all $s\in S$. %sorry mistake in the comments.
Using randomized optimization \cite{tadic2006randomized}, an equivalent representation of the set of constraints in \eqref{eq:value} is	\begin{equation*}
\Expect \limits_{ \bs \sim \Delta_1} B(g(\bs; \theta))=0,
\end{equation*}
where $\bs$ is a random variable with a distribution $\Delta_1$ whose support is $S$. Similarly, let
\[
\ell(s; \theta) = \Pr(D(s,K; \theta) \geq \alpha ) - \beta, \forall s\in Y.\] The equivalent representation of \eqref{eq:chance} is
\[
\Expect \limits_{ \bs\sim \Delta_2} B(\ell(\bs; \theta))=0,
\]
where  $\bs$ is a random variable with a distribution $\Delta_2$ over $Y$. 

Thus, \eqref{eq:adp} is equivalent to:
\begin{align}
\label{eq:rand-opt}
\begin{split}
\min_\theta         & \sum_{s\in S} c(s) V(s; \theta),                   \\
\mbox{subject to: } & \Expect_{\bs \sim \Delta_1} B(g(\bs; \theta))=0,      	 \Expect_{\bs \sim \Delta_2} B(\ell(\bs; \theta)) = 0. 
\end{split}
\end{align}

\noindent \paragraph*{The choice of weights: } We choose  the state relevant weight $c(s) = c(s; \theta)$ to be the frequency 
with which different states are expected to be visited in the chain under policy $\pi(\cdot;\theta)$, which is computed from $V(\cdot;\theta)$ using \eqref{eq:Q-policy}.
To justify the choice of this state relevant weight, it is noted that in the absence of chance constraints, we have the following optimization problem:	
\begin{align}
\label{eq:rand-opt-no-chance}
\begin{split}
\min_\theta         & \sum_{s\in S} c(s ; \theta) V(s; \theta),    
\mbox{ subject to: } \Expect_{\bs \sim \Delta_1} B(g(\bs; \theta))=0. 
\end{split}
\end{align}

The following result has been proved in \cite{de2003linear} and rephrased with softmax Bellman operator.
\begin{lemma}\cite{de2003linear}
A vector $\theta^\ast$  solves \begin{align*}
\begin{split}
\min_\theta & \sum_{s \in S} c(s)\Phi\theta  
\mbox{ subject to: }  \calB \Phi\theta  - \Phi\theta   \leq 0,
\end{split}
\end{align*}
if and only if it solves 
\begin{align*}
\begin{split}
\min_\theta         & \sum_{s \in S} \norm{V^\ast - \Phi\theta}_{1,c},
\mbox{ subject to: } \calB \Phi\theta  - \Phi\theta   \leq 0,
\end{split}
\end{align*}
where $ \norm{V^\ast - \Phi\theta}_{1,c} = \Expect_{\bs\sim c} \abs{ \Phi(\bs)\theta - V^\ast(\bs; \theta)}$.
\end{lemma}

For any state relevant weight $c\in \Delta(S)$, it holds that $\min_\theta  \sum_{s \in S} \norm{V^\ast - \Phi\theta}_{1,c} \ge \min_\theta  \sum_{s \in S} \norm{V^\ast - \Phi\theta}_{1,c(\cdot,\theta)}$ where $c(\cdot, \theta)$ is defined earlier. According to Theorem~1 of \cite{de2003linear}, the ideal weight is to choose $c$ that captures the (discounted) frequency
with which different states are expected to be visited. For a given parameter $\theta$, the weight function $c(\cdot, \theta)$ can be obtained from on-policy sampling of   trajectories with $\pi(\cdot;\theta)$.

After formulating the optimization problem into stochastic programming problem, the augmented Lagrangian function of \eqref{eq:rand-opt} is
\begin{align}
\begin{split}
& L_{\nu}(\theta, \lambda, \xi) = \sum_{s\in S} c(s; \theta) V(s; \theta) + \lambda \cdot \Expect\limits_{\bs\sim \Delta_1} B(g(\bs; \theta))      \\
& + \frac{\nu}{2} \cdot \abs{\Expect\limits_{\bs \sim \Delta_1}B(g(\bs; \theta))}^2 + \xi \cdot \Expect\limits_{\bs\sim \Delta_2} B(\ell(\bs; \theta)) \\
& + \frac{\nu}{2} \cdot \abs{\Expect\limits_{\bs\sim \Delta_2} B(\ell(\bs ; \theta))}^2,
\end{split}                                                       
\end{align}
where $\lambda$ and $\xi$ are the Lagrange multipliers and $ \nu$ is a large penalty constant.
Following the Quadratic Penalty function method  \cite{bertsekas1999nonlinear}, an optimal solution of \eqref{eq:rand-opt} consists of solving a sequence of $\emph{inner}$ optimization problems of the form: 
\begin{equation}	\label{eq:inner}
\min_{\theta \in \reals^\calK} \;    L_{\nu^k}(\theta, \lambda^k, \xi^k),
\end{equation}
where $\{\lambda^k\}$ and $\{\xi^k\}$ are sequences in $\reals$, $\{\nu^k\}$ is a positive penalty parameter sequence, and $\calK$ is the size of $\theta$. % previously used \calK

\subsection{Trajectory Sampling-based Approximate Value Iteration}
\label{sec:avi}
In this subsection, we devise a sampling-based method for solving \eqref{eq:rand-opt}. It is achieved by showing that the gradient of $L_{\nu^k}(\theta^k, \lambda^k, \xi^k)$ can be computed from sampled trajectories. 

With a slight abuse of notation, let $M^{\theta}$ be a Markov chain induced by policy $\pi(\cdot; \theta)$. By selecting $c(s;\theta) =\sum_{t=0}^\infty \Pr(X_t=s)$ the state visitation frequency in the Markov chain $M^\theta$, for an arbitrary function $f:S\rightarrow \reals$,  it holds that
\begin{equation}
\underset{s \in S}{\sum}c(s; \theta)f(s;\theta) = \int p(h; \theta)f(h;\theta)dh,
\end{equation}
where $p(h; \theta)$ is the probability of path $h$ in the Markov chain $M^\theta$, $f(h;\theta) = \sum_{i=1}^{|h|}f(s_{i};\theta)$.

Furthermore, by selecting $\Delta_1 \propto c(\cdot;\theta)$ and letting $f(s;\theta) = V(s; \theta) + \lambda^k \cdot B( g(s; \theta)) + \frac{\nu^k}{2} \cdot \abs{B(g(s; \theta))}^2$, and $m(\Delta_2; \theta) =\xi ^k \cdot \Expect\limits_{\bs\sim \Delta_2} B(\ell(\bs; \theta)) +\frac{\nu^k}{2} \cdot \abs{\Expect\limits_{\bs\sim \Delta_2} B(\ell(\bs ; \theta))}^2$, the $k$-th objective function  in \eqref{eq:inner} becomes
\begin{equation*}
\min_{\theta} \underbrace{ \int p(h; \theta)f(h)dh}_ {F(\theta)} + m(\Delta_2 ; \theta).
\end{equation*} 

Using the gradient descent,  parameter $\theta$ is updated by
\begin{equation*}
\theta^{j+1} \leftarrow \theta^{j}  - \eta_1 \cdot \nabla_{\theta} F(\theta)  - \eta_2 \cdot \nabla_{\theta} m(\Delta_2; \theta),
\end{equation*}
where $j$ represent the $j$-th inner iteration, $\eta_1 \text{ and } \eta_2$ are positive step sizes. 
\begin{equation*}
\nabla_\theta F(\theta)  = \int \underbrace{\nabla_\theta p(h; \theta) f(h; \theta) dh}_{1}  + \int \underbrace{ p(h; \theta) \nabla_\theta f(h; \theta)dh}_{2},
\end{equation*}
where
\begin{align*}
1=      & \int \nabla_\theta p(h; \theta) f(h; \theta)dh                                                                                \\
=       & \int p(h; \theta)\nabla_\theta \log p(h; \theta)f(h; \theta)dh                                                                \\
=       & \int p(h; \theta) \bigl[\sum_{t=0}^{\abs{h}} \nabla_\theta \log \pi(a_t \mid s_t; \theta)\bigr] f(h; \theta)dh                      \\
\approx & \frac{1}{N_h} \sum_{h \sim p(h; \theta)} \bigl[\sum_{t=0}^{\abs{h}} \nabla_\theta \log \pi(a_t \mid s_t; \theta)\bigr] f(h; \theta). \\
&(\text{Monte-Carlo approximation})
\end{align*}

Variance reduction \cite{williams1992simple} can be employed to reduce the error variance in the Monte-Carlo approximation.

% Since $f(h; \theta)$ term contains exponential value, which can worsen the variation and lead to the divergence of the algorithm. Using variance reduction \cite{williams1992simple}, we have
% \begin{align*}
% 1 \approx & \frac{1}{N_h} \sum_{h \sim p(h; \theta)} \bigl[\sum_{t=0}^{T} \nabla_\theta \log \pi(a_t \mid s_t; \theta)\bigr] f(h; \theta)       \\
% =         & \frac{1}{N_h} \sum_{h \sim p(h; \theta)} \bigl[\sum_{t=0}^{T} \nabla_\theta \log \pi(a_t \mid s_t; \theta)\bigr] (f(h; \theta) - b) 
% \end{align*}
% where $b$ is a constant.
\begin{align*}
2 = & \int p(h; \theta) \nabla_\theta f(h; \theta)dh                           
= \int p(h; \theta) [\sum_{t = 0}^{\abs{h}} \nabla_\theta f(s_t; \theta)] dh \\
\approx & \frac{1}{N_h} \sum_{h \sim p(s_t; \theta)} \bigl[\sum_{t = 0}^{\abs{h}} \nabla_\theta f(s_t; \theta) \bigr],
\end{align*}
where $ \nabla_\theta f(s_t; \theta)= \nabla_{\theta} V(s_t; \theta) +\lambda^k \cdot \nabla_{g}B(g(s_t; \theta)) \nabla_{\theta}g(s_t; \theta) + \nu^k \cdot B(g(s_t; \theta) \nabla_{\theta}g(s_t; \theta).$	
Note if $B(x) = \max\{x,0\}$, then  $\nabla_{g}B(x) = \begin{cases} 1 &\mbox{if } x> 0, \\ 0 & \mbox{otherwise.} \end{cases}$

In term of the derivative of the second term:
%m(\Delta_2; \theta) =\xi \cdot \Expect\limits_{s\sim \Delta_2} B(\ell(s; \theta)) +\frac{\nu}{2} \cdot \abs{\Expect\limits_{s\sim \Delta_2} B(\ell(s ; \theta))}^2
\begin{align*}
\nabla_{\theta} m (\Delta_2; \theta) & = \xi^{k} \cdot \Expect\limits_{\bs\sim \Delta_2} \nabla_{\ell}  B(\ell(\bs; \theta)) \nabla_{\theta} \ell(\bs; \theta)                \\
& + \nu^{k} \Expect\limits_{\bs\sim \Delta_2} \cdot B(\ell(\bs; \theta)) \nabla_{\ell}  B(\ell(\bs; \theta)) \nabla_{\theta}\ell(\bs ; \theta).
\end{align*}

We generate a set $Z$ of trajectories starting at a state with the initial distribution $\Delta_2$ to estimate the gradient. A simple choice of $\Delta_2$ is a uniform distribution over $Y$. Let $Z(s)$ be a set of trajectories in $Z$ with the initial state $s$ and $D(z)$ be the total cost along the trajectory $z$, we have
\begin{align*}
\begin{split}
\nabla_{\theta}\ell(s; \theta)
& \approx \nabla_{\theta}\left[ \sum_{z\in Z(s)} p(z; \theta ) \indicator \{D(z) \geq \alpha \} - \beta\right]      \\
& = \sum_{z\in Z(s)} \nabla_{\theta}p(z; \theta ) \indicator \{D(z) \geq \beta \}                                        \\
& = \sum_{z \in Z(s)} p(z; \theta \mid s) \nabla_{\theta} \log{p(z; \theta )} \indicator \{D(z) \geq \alpha \}             \\
& \approx \frac{1}{|Z(s)|} \sum_{z\in Z(s)} \nabla_{\theta} \log{p(z; \theta  )} \indicator \{D(z) \geq \alpha \},            \\
%& = \frac{1}{|Z(s)|} \sum_{z\in Z(s)}[\sum_{t=0}^{\abs{z}} \nabla_\theta \log \pi(a_t \mid s_t; \theta)] \cdot \indicator \{D(z) \geq \alpha \}.
\end{split}
\end{align*}
where $\nabla_{\theta} \log p(z; \theta  ) = \sum_{t=0}^{\abs{z}} \nabla_\theta \log \pi(a_t \mid s_t; \theta)$.
% & = \frac{1}{|Z|} \sum_{z}\sum_{t=0}^{T} \frac{1}{\tau}[\phi(s_t,a_t) - \nabla_{\theta}V(s; \theta)]\indicator \{D(s) \geq \alpha \}
\begin{align*}
\begin{split}
B(\ell(s; \theta))
& \approx B\left[ \sum_{z\in Z(s)} p(z; \theta  ) \indicator \{D(z)\geq \alpha \} - \beta\right] \\
& = \max (\sum_{z\in Z(s)} p(z; \theta \mid s) \indicator \{D(z)\geq \alpha \} - \beta, 0)            \\
& \approx \max ( \frac{1}{|Z(s)|} \sum_{z\in Z(s)}\indicator \{D(z)\geq \alpha \} - \beta, 0).         
\end{split}   
\end{align*}

Finally, the  gradient $	\nabla_{\theta} m (\Delta_2; \theta) $ is approximated by Monte Carlo approximation. 		The sample bound $\abs{Z}$ is determined by methods in stochastic programming \cite{de2004constraint}. 
Using the gradient descent algorithm we can update $\theta^{j}$ to $\theta^{j+1}$ until $\norm{\theta^{j+1} - \theta^{j}} \leq \epsilon$.  The step-sizes are updated using  the square summable step rule  \cite{boyd2003subgradient}, that is,
$	\eta_{i}^{k+1} = \frac{\eta_{i}^{k}}{k}, \quad \mbox{for }i=1,2.
$

After the inner optimization for \eqref{eq:inner} converges,  we update formula for multipliers $\lambda$ and $\xi$ as
\begin{align*}
& \lambda^{k+1} = \lambda^{k} + \nu^{k} \cdot  \Expect\limits_{\bs\sim \Delta_1} B(g(\bs;\theta^k)); \\
& \xi^{k+1} = \xi^{k} + \nu^{k} \cdot \Expect\limits_{\bs\sim \Delta_2} B(\ell(\bs; \theta^k)),
\end{align*}
where the expectations are approximated by Monte-Carlo methods. The penalty value $ \nu$ is updated using the rule in \cite{bertsekas1999nonlinear}.

\resizebox{0.48\textwidth}{!}{
\begin{minipage}{\linewidth}
\begin{align*}
& \nu^{k+1} = \begin{cases}b \nu^{k}
&\norm{\Expect_{\bs\sim \Delta_2} B(\ell(\bs; \theta^{k})} > \rho \norm{\Expect_{\bs\sim \Delta_2} B(\ell(\bs; \theta^{k-1})} \\
& \mbox{or } \\
& \norm{\Expect_{\bs\sim \Delta_1} B(g(\bs; \theta^{k})} > \rho \norm{\Expect_{\bs\sim \Delta_1} B(g(\bs; \theta^{k-1})} \\ \nu^{k} &\mbox{otherwise,}\end{cases}
\end{align*}
\end{minipage}
}
where $b>1$ to ensure  the sequence  $\nu^{k}$ are non-increasing and a typical choice of $\rho$ is 0.25. Note that  a different penalty parameter $\nu^k_i$ can be chosen for each constraint. For example, $\nu_1^k$ is chosen for constraint \eqref{eq:value} and $\nu^k_2$ is chosen for constraint \eqref{eq:chance}. Furthermore, the update of penalty parameter does not require any new samples, since  $\Expect_{\bs\sim \Delta_1} B(g(\bs; \theta^{k}) $ and $\Expect_{\bs \sim \Delta_2} B(\ell(\bs; \theta^{k}) $ can be evaluated based on sampled trajectories.

The following assumptions are required in the analysis of the convergence.
\begin{itemize}
\item[A1] For $y\in \{\eta_1,\eta_2\}$ , $y^k>0$ for each $k \ge 1$, $\sum_{k=1}^\infty y^k=\infty$, and $\sum_{n=1}^\infty (y^k)^2 < \infty$.
\item[A2] The value function approximation is continuously differentiable in $\theta$ and $\nabla_\theta V(\cdot;\theta)$ is locally  Lipschitz continuous.
\item [A3] There exists a feasible solution for \eqref{eq:rand-opt}.
\item [A4] For time-unbounded \ac{pctl}, the length of a sampled trajectory is lower bounded by either the time bound in \ac{pctl} formulas or the mixing time of the Markov chain with policy parameterized by $\theta$.
\end{itemize}
\begin{theorem} 
Assuming A1-A4, the sequence of value function updates converges almost surely (with probability 1) to a local optimal solution $\theta^\ast$ for the chance-constrained optimization problem in \eqref{eq:rand-opt}.
\end{theorem}
The convergence proof is standard for stochastic programming and omitted for space limitation.

% \begin{remark}
% 	Note that the algorithm naturally extends to handle multiple chance constraints, which can be resulted from having conjunction of several \ac{pctl} formulas.
% \end{remark}

\section{Case Studies}\label{sec:exp}
We validate the algorithm in two motion planning problems modeled as stochastic gridworld problems, illustrated in Fig.~\ref{fig:simulation-two}: A simple reach-avoid task (a), and a planning problem with a \ac{pctl} constraint (b). In each state $s \in S$ and for robot's different actions (heading up (`U'), down (`D'), left (`L'), right (`R')), the probability of arriving at the correct cell is $1 - 0.1 \times N$, where $N$ is the number of the neighbors of the current state including itself. If the system hits the wall, it will be bounced back to its original cell.

\subsection{Planning without \ac{pctl} constraints} \label{exp:1} Without considering any constraints, the first experiment is designed to observe the
the relation between the weighting parameters and the approximation
error and to justify the choice of weights in \eqref{eq:adp}. The planning objective is to find an approximately optimal policy which drives the robot from
the initial position $s_{init}:[0,0]$ to the goal $s_{goal}: [8, 10]$ while
avoiding the obstacles. The reward is defined as the following: the robot receives a reward of 100 if $P(s_{goal} \mid s, a) > 0.5$.

The value function approximation is $V(s; \theta) = \Phi \theta$, where the basis functions $\Phi= [\phi_1,\phi_2,\ldots, \phi_{\calK}]^\intercal$ are \ac{ggks} \cite{sugiyama2015statistical}
defined as the following:
$ \Phi_{j}(s)  = K(s, c^{(j)}) $ 
and $K(s, s') = \exp(-\frac{SP(s,s')^2}{2\sigma^2}), $
where $\{c^{(j)}, j=1,\ldots, \calK\} $ is a set
of preselected centers. In this example, we select
the centers to be $\{(x,y)\mid x,y \in \{0,5,10\}\}$ and the variance $\sigma$ to be 5. The term $SP(s, s')$ refers to the shortest path from state $s$ to state $s'$ in the graph, assuming deterministic transitions. 
\begin{figure}
\vspace*{0.5cm}
\centering
\begin{subfigure}[h]{0.22\textwidth}
\includegraphics[width=\textwidth]{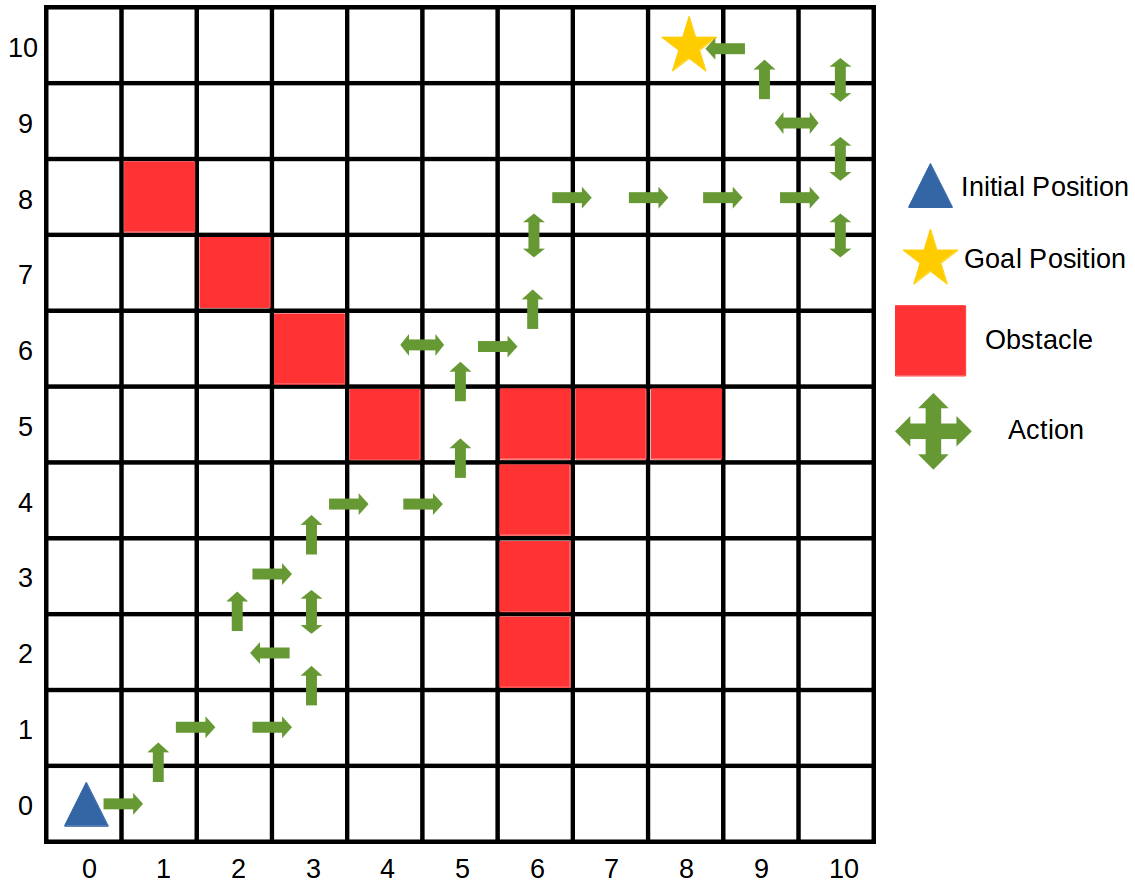}
\caption{}
\label{fig:simulation}
\end{subfigure}
\begin{subfigure}[h]{0.22\textwidth}
\includegraphics[width=\textwidth]{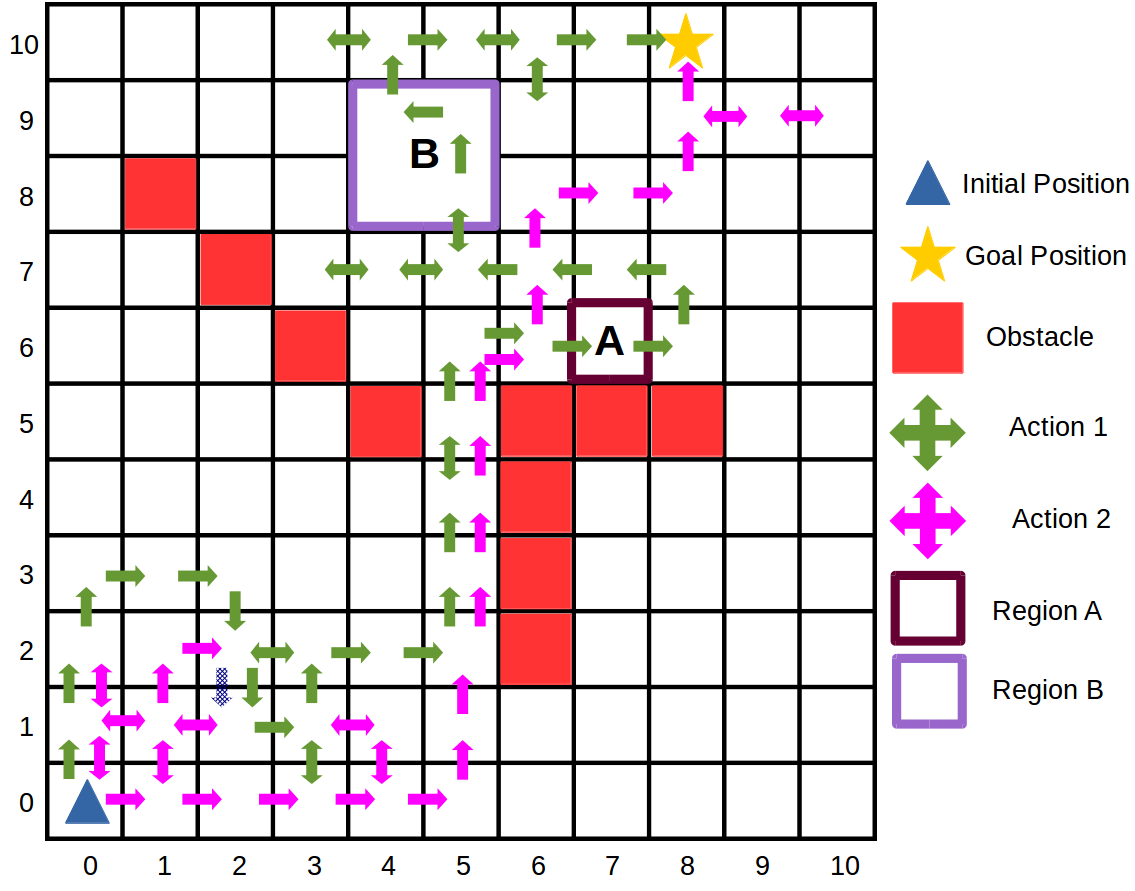}
\caption{}
\label{fig:simulation_c}
\end{subfigure}
\caption{(a) Gridworld with initial (triangle), goal (star), obstacles(solid squares), no \ac{pctl} constraints.  (b) The gridworld with \ac{pctl} constraint  $A \implies \Pr_{\geq \delta}(C_{\leq 13} \Eventually^{\le 15} B)$ where $A$ and $B$ are regions marked in the graph.}
\label{fig:simulation-two}
\end{figure}

The parameters used in the algorithm are the following (with respect to Section~\ref{sec:avi}):  the
temperature parameter $\tau=5$, $b=1.1$, $\eta_1=0.1$, the initial penalty parameter $\nu^0=10.0$, the initial Lagrangian multipliers $\lambda^0= 0$.  During each iteration, $30$ trajectories of length $\le 6$ are sampled. % The maximum iteration number for the outer loop is $10$, which defines how many sequential problems defined by Quadratic Penalty Function method need to be solved, and the inner loop is $200$ in each optimization problem.
The stopping criterion for each inner optimization problem is $\norm{\nabla_{\theta}L_{\nu^k}(\theta^j, \lambda^k )} \leq \epsilon^{k} $, where $\{\epsilon^k\}$ is a positive sequence converging to $0$. In a single run,
the algorithm converges after $4$ outer
iterations, with $164, 24, 8, 1$ iterations for
each outer iteration, respectively.
\begin{figure}[t!]
\vspace*{-0.5cm}
\begin{subfigure}[h]{0.24\textwidth}
\includegraphics[trim=1cm 0cm 3.2cm 0cm, clip=true, width=\textwidth]{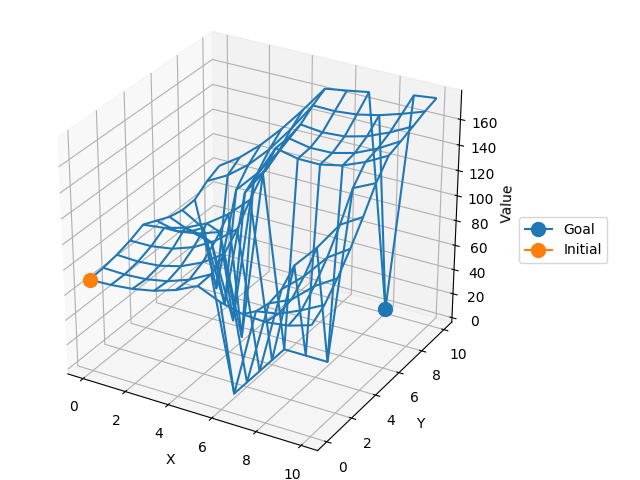}
\caption{}
\label{fig:surf_el}
\end{subfigure}
\begin{subfigure}[h]{0.24\textwidth}
\includegraphics[trim=1cm 0cm 3.2cm 0cm, clip=true, width=\textwidth]{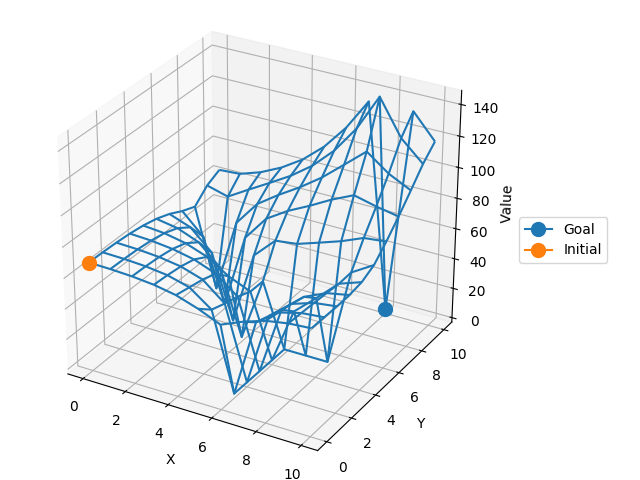}
\caption{ }
\label{fig:surf_vi}
\end{subfigure}
\caption{Value Functions. (a) Approximate value function obtained with the proposed \ac{adp} method. (b) The true value function obtained with softmax value iteration.} %The crosses represents obstacles. }
\label{fig:surf}
\vspace{-4ex}
\end{figure}
Fig.~\ref{fig:statistic_result} shows the result of $100$ independent experiments for solving the same planning problem starting with the same initial $\theta$ which is a zero vector. The black line represents the mean, the shaded area is limited by the maximum and minimum values over iterations, and the red line is the ground truth. The black line is always above the red line. This is expected as the value function approximation is an upper bound of the optimal value function.
\begin{figure}
\centering
\includegraphics[width=0.35\textwidth]{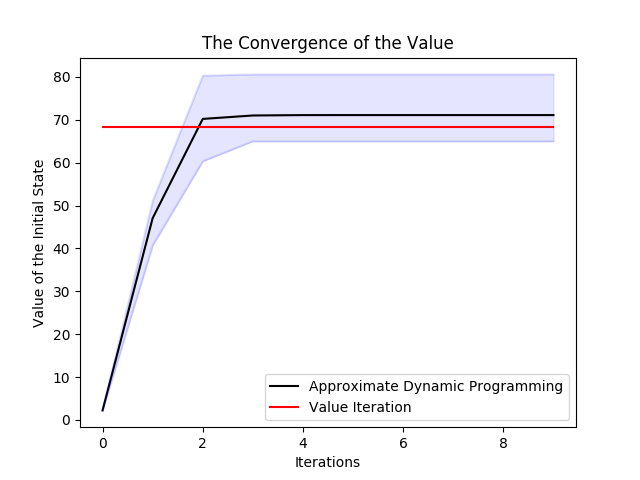}
\caption{The learning curve for the stochastic gridworld, averaged across 100 runs of the \ac{adp} algorithm with random initialization. }
\label{fig:statistic_result}
\end{figure}

To illustrate the approximation error, we compare the
optimal   and   approximate-optimal value
functions  in Fig.~\ref{fig:error_surf}, which plots the error $V(s; \theta) - V(s)^\ast$ on each individual state $s$. Fig.~\ref{fig:frequency_surf} shows state visitation frequency under the computed policy. It shows that for a state with a high visitation frequency under the optimal policy, the error tends to be very small. This result is expected due to our choice of weight parameters that have larger weights for states with high visitation frequencies.	Fig.~\ref{fig:simulation} shows one run generated by following the computed policy.
% grammar check!
\begin{figure}[t!]
\vspace*{-0.5cm}
\centering
\begin{subfigure}[h]{0.22\textwidth}
\includegraphics[trim=2cm 0cm  2cm 0cm, clip=true, width=\textwidth]{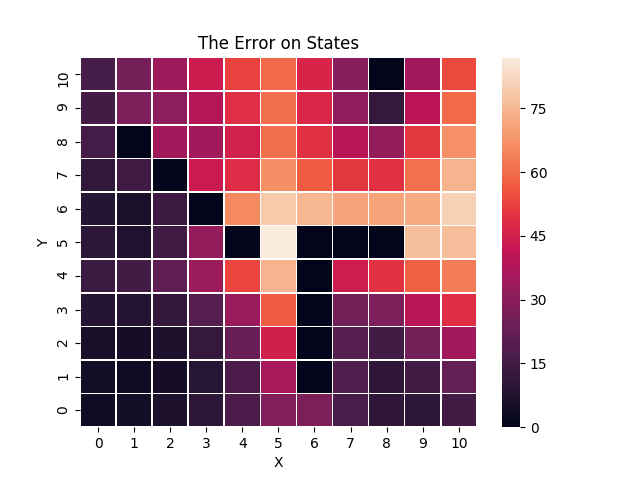}
\caption{}
\label{fig:error_surf}
\end{subfigure}\quad
\begin{subfigure}[h]{0.22\textwidth}
\includegraphics[trim=2cm 0cm  1.5cm 0cm, clip=true, width=\textwidth]{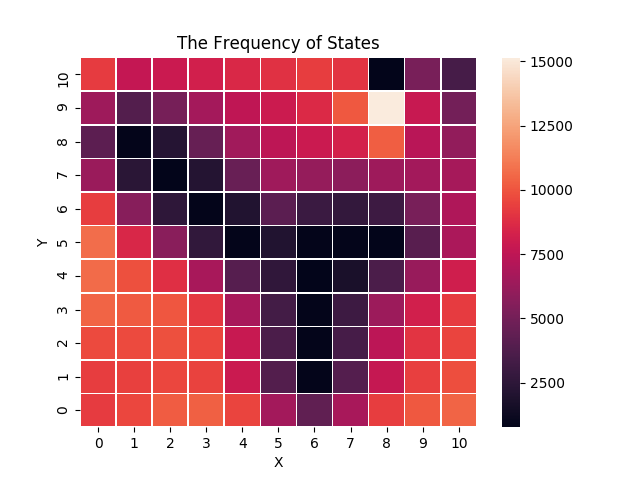}
\caption{}
\label{fig:frequency_surf}
\end{subfigure}
\caption{Comparison of state visitation frequency and the approximation error for different states (a) The heatmap of the error $V(s;\theta^\ast)- V^\ast(s)$. (b) The heatmap of state visitation frequencies under policy $\pi(\cdot;\theta)$.}
\vspace{-4ex}
\label{fig:error_frequency}
\end{figure}

\subsection{Planning with \ac{pctl} constraints} \label{exp:2}Consider  including a \ac{pctl} constraint $ A \implies \Pr_{\geq \delta}(\Next C_{\leq 13} \Eventually^{\le 14} B) $: When the agent visits $A$, then it will ensure, starting from the next state, with a probability at least $0.2$, to eventually visit region $B$ in less than 14 steps with a cost less than 13. Region A and B are shown in Fig.~\ref{fig:simulation_c}. Let $d: S \times A \rightarrow \reals$ be defined by $d(s,a) = 1$.

We use the same value function approximation, the same stopping criterion for the inner optimization, and the same set of parameters with different initial penalty parameters $\nu_1^0= 10.0$ and $\nu_2^0= 500$ for constraints \eqref{eq:value} and \eqref{eq:chance}, respectively. % Initial $\lambda$ and $\xi$ are $1.0$.
\begin{figure}
\centering
\includegraphics[width=0.35\textwidth]{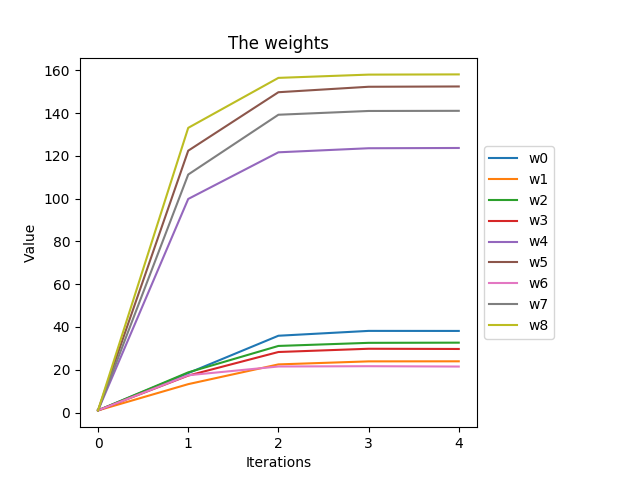}
\caption{The Convergence of Parameters}
\label{fig:weights}
\end{figure}
For each iteration, there are $30$ trajectories (of lengths $\le 6$) sampled and another $100$ trajectories (of lengths $= 15$) starting from the region A for the chance constraints. %The maximum iteration number for the outer loop is $10$ and the inner loop is $200$.
Given $\delta=0.2$, Fig.~\ref{fig:weights} shows the convergence of the parameters. Even by adding \ac{pctl} constraints, the algorithm converges after $4$ outer iterations. Fig.~\ref{fig:simulation_c} shows two trajectories simulated by following the computed policy. In this one sampled trajectory, the system reaches A and then B with a cost less than the given threshold. In another sampled trajectory, the system does not visit region A and directly goes to the goal.

We tested the algorithm with different values for the $\delta$. 
Table.~\ref{tb:different_delta} shows the frequencies of trajectories satisfying the cost constraint $\Next C_{\leq 13} \Eventually^{\le 14} B$ under different values for $\delta$ after the empirical evaluation of $20000$ trajectories starting from $A$. In all experiments, the \ac{pctl} constraint is satisfied. 
The result shows that as $\delta$ increases, the probability of satisfying the constraint also increases, but not monotonically. This can be caused by the chosen function approximation.
\begin{table} 
\caption{Frequencies of satisfying paths under different $\delta$.}
\label{tb:different_delta}
\begin{center}
\begin{tabular}{ccccc}
\hline			
$\delta$ & 0 & 0.1 & 0.2 & 0.3 \\ 
Num. of satisfying  paths & 3130& 2227 & 4581 & 7410 \\  
\% of satisfying paths & $0.15$ & $0.11$ & $0.23$ & $ 0.37$ \\ 
\hline			
\end{tabular}\vspace{-4ex}
\end{center}
\end{table}

% It is worth mentioning that most of the computation time is spent on simulating the trajectories given by the fixed time steps. However, by introducing the parallel computing, the simulation will greatly speed up, since sampling the initial states is entirely random in other words, along any action sequences we perform the minimization regardless of provenance, and the trajectories are just the evolution of the stochastic system dynamics.

% the discussion may be saved for journal. or toolbox.

\section{conclusion}
\label{sec:con}
We have presented an approximate value iteration method for \ac{mdp} with \ac{pctl} constraints. We proposed a method that translates \ac{pctl} constraints into chance constraints and uses stochastic programming for solving an upper bound of the optimal value function subject to constraints in \ac{pctl}. The almost sure convergence of the proposed algorithm is guaranteed under several assumptions.  There are several future directions enabled by this study: First, the current method only studies a class of \ac{pctl} for which memoryless policies are sufficient to be approximately optimal. We are interested in extending this method to a large class of temporal logic formulas, for which finite-memory is needed for optimality. Second, it is possible to develop distributed \ac{adp} using approximate value iteration based on decomposition-based planning in large-scale \ac{mdp}s. 

% There are a number of interesting future directions. First, we are thinking about extending our work into continuous state space and action space. Also, we are trying to incorporate more complicated system-level specifications, \ie, \ac{ltl}. Moreover, we only consider one agent, in the future, we want to extend this work into multi-agent applications.

% \begin{remark}
% 	It is noted that the proposed formulation does not handle
% 	other nested formulas, including $P$-operator within a
% 	$P$-operator, or $P$-operator inside a $C$-operator, or
% 	nesting of more than three $P$- or $C$- operators. %  In
% 	% \cite{lahijanian2012temporal}, the authors addressed the
% 	% planning problem with $P$- operator nesting within another
% 	% $P$-operator but does not consider $C$-operator. 
%         We plan to
% 	investigate the planning problem with multiple $P$- and $C$-
% 	operators nesting in our future work.
% \end{remark}

% \section*{Acknowledgment}
% We would like to thank Xuan Liu for his comments.
% ADD ACKNOWLEDGEMENT IF IT IS ACCEPTED.
%	\appendices
%	\input{appendix.tex}

%\bibliographystyle{plain}
 \bibliographystyle{ieeetr}
\bibliography{refs}

\end{document}